\documentclass[12pt]{amsart}
\usepackage[margin=3cm]{geometry}
\usepackage{amsmath,amssymb,amsthm,amscd}
\usepackage{mathtools} 
\usepackage{enumitem}
\usepackage{graphicx}
\usepackage{verbatim}
\usepackage{caption}
\usepackage[usenames]{xcolor}\definecolor{darkgray}{rgb}{.412,.412,.412}
\usepackage{booktabs,tabularx,ragged2e}
\newcolumntype{Y}[1]{>{\centering}m{\dimexpr#1\textwidth-2\tabcolsep-1\arrayrulewidth\relax}}
\newcolumntype{Z}{@{}m{0pt}@{}} 

\newcommand{\T}{\mathcal{T}}
\newcommand{\M}{\mathcal{M}}
\newcommand{\Sigmahat}{\hat{\Sigma}}
\DeclareMathOperator{\Crit}{Crit}
\DeclareMathOperator{\Flat}{Flat}
\DeclareMathOperator{\len}{len}
\newcommand{\id}{\textnormal{id}}
\DeclareMathOperator{\Follower}{Fol}
\newcommand{\fol}[1]{\Follower(#1)}

\theoremstyle{plain}
\newtheorem{theorem}{Theorem}
\newtheorem*{thm1}{Theorem \ref{th:preimages}}
\newtheorem*{thm2}{Theorem \ref{th:equicontinuous}}
\newtheorem*{thm3}{Theorem \ref{th:main}}
\newtheorem{lemma}[theorem]{Lemma}
\newtheorem*{lem9'}{Lemma \ref{lem:perturbcrit}$'$}
\newtheorem*{lem10'}{Lemma \ref{lem:onto}$'$}
\newtheorem*{lem11'}{Lemma \ref{lem:ontopert}$'$}
\newtheorem*{lem12'}{Lemma \ref{lem:unileo}$'$}
\newtheorem*{lem13'}{Lemma \ref{lem:equiunileo}$'$}
\newtheorem{proposition}[theorem]{Proposition}
\newtheorem{corollary}[theorem]{Corollary}
\theoremstyle{definition}

\newtheorem*{question}{Question}

\newtheorem{remark}[theorem]{Remark}
\newtheorem{example}[theorem]{Example}

\begin{document}

\title{Constant Slope Models and Perturbation}
\author{Michal Malek}
\author{Samuel Roth}
\date{}
\maketitle

\begin{abstract}
We sharpen an estimate for the growth rate of preimages of a point under a transitive piecewise monotone interval map. Then we apply our estimate to study the continuity of the operator which assigns to such a map its constant slope model.
\end{abstract}

\section{Introduction}

\subsection*{Motivation}\strut\\
Over 50 years ago, W.~Parry showed that each continuous topologically transitive piecewise monotone interval map is conjugate by an increasing homeomorphism to a map with constant slope. Recently, Ll.~Alsed\`{a} and M.~Misiurewicz pointed out that this constant slope model is unique, and thus it makes sense to study the operator $\Phi$ which assigns to a map its constant slope model. They showed that $\Phi$ is not continuous, essentially because $C^0$ perturbation of the map can lead to a jump in topological entropy. Nevertheless, they conjectured that within the space of transitive maps of a fixed modality, $\Phi$ is continuous at each point of continuity of the topological entropy \cite[page 13]{AM}. This paper confirms that conjecture.

\subsection*{Definitions}\strut\\
We study maps $f:[0,1]\to[0,1]$ which are continuous and \emph{piecewise monotone}, i.e. with a finite set of \emph{critical points}:
\begin{equation*}
\Crit(f)=\{0,1\} \cup \{x \,|\, f \text{ is not monotone on any neighborhood of }x\}.
\end{equation*}
The \emph{modality} of $f$ is the cardinality of $\Crit(f)\cap(0,1)$. We are interested in the spaces
\begin{itemize}
\item $\T$ - the space of topologically transitive piecewise monotone maps,
\item $\T_m$ - the subspace of transitive maps of a fixed modality $m\in\mathbb{N}$, and
\item $\mathcal{H}^+$ - the space of monotone increasing homeomorphisms of $[0,1]$ with itself.
\end{itemize}
Each of these spaces is contained in $\mathcal{C}^0$, the space of all continuous functions from $[0,1]$ to itself. Moreover, we equip our spaces with the topology of uniform convergence given by the usual $C^0$ metric
\begin{equation*}
d(f,g)=\max_{x\in[0,1]} |f(x)-g(x)|.
\end{equation*}
A map $f\in\T$ has \emph{constant slope} $\lambda$ if $|f'(x)|=\lambda$ for $x\notin\Crit(f)$. We say that $\tilde{f}$ is a \emph{constant slope model} for $f$ if $\tilde{f}$ has some constant slope $\lambda$ and there is a homeomorphism $\psi\in\mathcal{H}^+$ such that $f=\psi\circ\tilde{f}\circ\psi^{-1}$.  It is known \cite[Theorem 8.2 and Corollary 1]{AM} that each map $f\in\T$ has a unique constant slope model, the conjugating homeomorphism is likewise unique\footnote{This is not explicitly stated in \cite{AM}, but is an easy corollary of \cite[Theorem 8.2]{AM}. If we have two conjugating homeomorphisms $\tilde{f}=\psi^{-1}\circ f \circ \psi = \phi^{-1} \circ f \circ \phi$, then $\psi^{-1}\circ\phi\circ\tilde{f}\circ\phi^{-1}\circ\psi=\tilde{f}$, so that by \cite[Theorem 8.2]{AM}, $\psi^{-1}\circ\phi=\id$, that is, $\psi=\phi$.}, and the constant slope is the exponential of the topological entropy of $f$. Thus, we are interested in two operators and one real-valued function on the space $\T$, namely
\begin{alignat*}{2}
\Phi&:\T\to\T, &\qquad \Phi(f)&=\text{the constant slope model for }f,\\
\Psi&:\T\to\mathcal{H}^+, & \Psi(f)&=\text{the conjugating homeomorphism, and}\\
h&:\T\to\mathbb{R}, & h(f)&=\text{the topological entropy of }f.
\end{alignat*}
Since conjugacy preserves modality, $\Phi$ preserves the spaces $\T_m$. We denote the restrictions of our operators to these spaces by $\Phi_m:\T_m \to\T_m$, $\Psi_m:\T_m \to\mathcal{H}^+$.

The role of these operators is summarized in the following commutative diagram.
\begin{equation*}
\begin{CD}
[0,1] @>\Phi(f)>> [0,1]\\
@V{\Psi(f)}VV @VV{\Psi(f)}V\\
[0,1] @>>f> [0,1]
\end{CD}
\end{equation*}
\subsection*{Results}\strut\\
We start with a theorem concerning the growth rate of the number of iterated preimages of an arbitrary point $x\in[0,1]$ under a transitive, piecewise monotone map $f$. It is already known that the exponential growth rate of the sequence $\left(\#f^{-n}(x)\right)$ gives the entropy of $f$, \cite[Theorem 1.2]{MR}. We show that the ``subexponential part'' of this sequence does not converge to zero.

\begin{theorem}\label{th:preimages}
Let $f\in\T$ and fix $x\in[0,1]$. Then
\begin{equation*}
\limsup_{n\to\infty} \frac{\# f^{-n}(x)}{e^{n h(f)}} > 0.
\end{equation*}
\end{theorem}

Our second result allows us to verify that a family of homeomorphisms in $\mathcal{H}^+$ is an equicontinuous family.

\begin{theorem}\label{th:equicontinuous}
If $K$ is a compact subset of $\T_m$, then $\Psi_m(K)$ is an equicontinuous family.
\end{theorem}

In the spirit of ``dynamical topology,''%
\footnote{The term dynamical topology was coined in \cite{KMS}. Simply put, it means we investigate topological properties of spaces of maps and operators on those spaces which are defined in terms of their dynamical properties.}
we may also state this result in purely topological terms. In light of the Arzela-Ascoli theorem, this says that the $\Psi_m$ image of a compact set is precompact, i.e., has a compact closure in $\mathcal{C}^0$.

These two theorems allow us to prove the main result of our paper, namely,
\begin{theorem}\label{th:main}
If a sequence of maps $g_n\in\T_m$ converges uniformly to $f\in\T_m$ and if $h(g_n)\to h(f)$, then the constant slope models $\Phi(g_n)$ converge uniformly to $\Phi(f)$.
\end{theorem}

As a corollary, we get a positive answer for the conjecture of Alseda and Misiurewicz, 
\begin{corollary}
The operator $\Phi_m$ is continuous at each continuity point of $h|_{\T_m}$. In particular, $\Phi_m$ is continuous for $m\leq 4$.
\end{corollary}

\subsection*{Sharpness of the Results}\strut\\
We remark that all of the hypotheses in Theorem~\ref{th:main} are essential. This is illustrated by the following two examples.

\begin{example} \emph{Perturbation with a jump in modality.}\\
\footnotesize
For each value $0\leq t\leq\frac{1}{4}$, put $a=\frac12- t$, $b=\frac12+ t$, $\lambda=3+2t$, and let $\tilde{g}_t$ be the (unique) map with constant slope $\lambda$, 10 critical points $0=c_0<c_1<\cdots<c_9=1$, and critical values $g(c_0)=a$, $g(c_2)=a+t^2$, $g(c_3)=0$, $g(c_4)=b+t^2$, $g(c_5)=a-t^2$, $g(c_6)=1$, $g(c_7)=b-t^2$, and $g(c_9)=b$. For $t>0$ let $\psi_t$ be the ``connect-the-dots'' map with dots at $(0,0)$, $(a,t)$, $(b,1-t)$, and $(1,1)$. Then put $g_t=\psi_t\circ\tilde{g}_t\circ\psi^{-1}_t$. Finally, let $f$ be the full 3-horseshoe, i.e. the ``connect-the-dots'' map with dots $(0,0)$, $(\frac13,1)$, $(\frac23,0)$, and $(1,1)$. As $t\to0$ we have uniform convergence of modality-8 maps to a modality-2 map $g_t \rightrightarrows f$ and convergence of entropy $h(g_t)=\log(3+2t) \to h(f)$, but the constant slope models converge to the ``wrong'' limit $\Phi(g_t)=\tilde{g}_t \rightrightarrows \tilde{g}_0 \neq f=\Phi(f)$. We omit the proofs of transitivity (when $t>0$) and uniform convergence -- these proofs are tedious but routine calculations, since all maps involved are piecewise affine.\\
\begin{tabular}{Y{0.2} Y{0.2} Y{0.2} Y{0.2} Y{0.2}}
\tiny $g_t$ &
\tiny $f$ & 
\tiny $\psi_t$ &
\tiny $\tilde{g}_t$ & 
\tiny $\tilde{g}_0$\tabularnewline
\includegraphics[width=2.7cm]{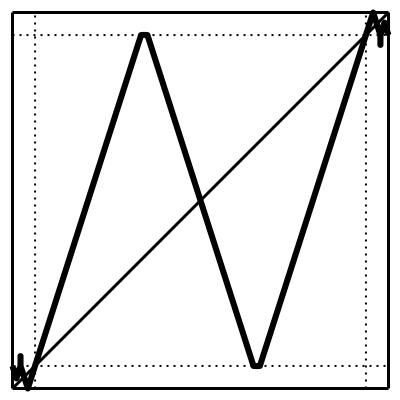} &
\includegraphics[width=2.7cm]{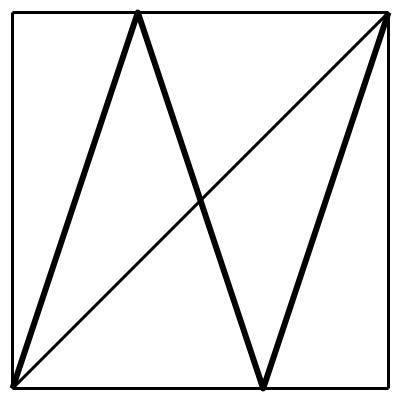} &
\includegraphics[width=2.7cm]{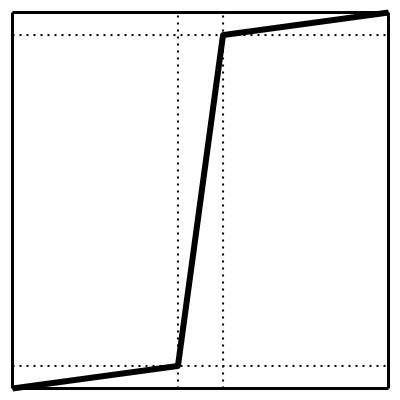} &
\includegraphics[width=2.7cm]{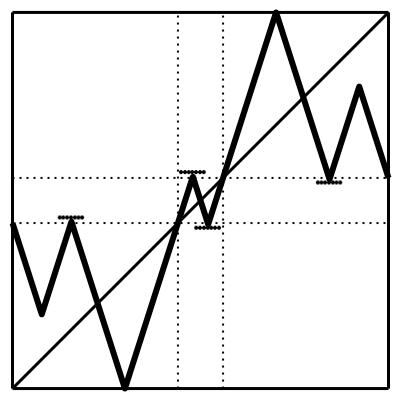} &
\includegraphics[width=2.7cm]{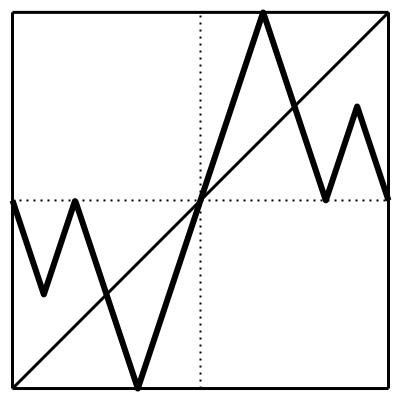} \tabularnewline
\end{tabular}\\[.8em]
\begin{minipage}{.60\textwidth}
Another perspective: when the map $\Phi$ carries the arc $(g_t)_{t\in[0,\frac14]}$ (writing $g_0=f$ for the endpoint) from $\T$ to the space of constant slope models $\Phi(\T)$, it ``breaks off'' the endpoint. Notice that the point $\tilde{g}_0$ is not even in $\Phi(\T)$ because it is not transitive.
\end{minipage}\hspace{.05\textwidth}%
\begin{minipage}{.35\textwidth}
\includegraphics{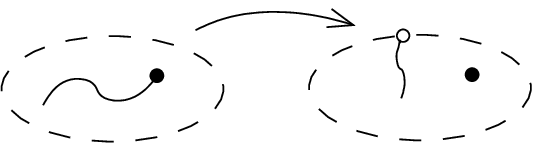}
\begin{picture}(0,0)
\put(75,53){\scriptsize$\Phi$}
\put(28,4){\scriptsize$\T$}
\put(115,4){\scriptsize$\Phi(\T)$}
\put(47,26){\tiny$f$}
\put(130,23){\tiny$\Phi(f)$}
\put(117,48){\tiny$\tilde{g}_0$}
\end{picture}
\end{minipage}
\end{example}

\begin{example} \emph{Perturbation with a jump in entropy.}\\
\footnotesize
Let $f:[0,72]\to[0,72]$ be the ``connect-the-dots'' map with dots $(0,32)$, $(20,52)$, $(24,60)$, $(25,58)$, $(32,72)$, $(52,32)$, $(58,20)$, $(60,24)$, $(72,0)$. The map has modality $5$. It was introduced in~\cite{M2} as a point of discontinuity of $h:\T_5\to\mathbb{R}$. Numerical calculations give $h(f)\approx \log 1.81299$, while~\cite{M2} shows rigorously that $h(f)<\log 2$. On the other hand, $f$ has a 2-cycle consisting of critical points $24\mapsto 60\mapsto 24$. Form $g_t$ by perturbing $f$ on the $t$-neighborhoods of those critical points, increasing the slope from 2 to 3 as in the figure below. As $t\to 0$ we have uniform convergence $g_t \rightrightarrows f$ within the space $\T_5$ (the proof of transitivity is omitted). On the other hand, $h(g_t)\geq\log2$ because $g_t^2$ has a 4-horseshoe. Thus, we may be sure that $\Phi(g_t)\not\rightrightarrows\Phi(f)$, because the slopes do not converge to the slope of $\Phi(f)$, see~\cite[Lemma 8.3]{M2}.
\begin{center}
\begin{tabular}{Y{.25}Y{.05}Y{.25}Y{.05}Y{.25}}
\tiny A 2-cycle of critical points &&
\tiny Perturbation near the 2-cycle &&
\tiny The 2nd iterates near\\one point of the 2-cycle \tabularnewline
\includegraphics[width=3.3cm]{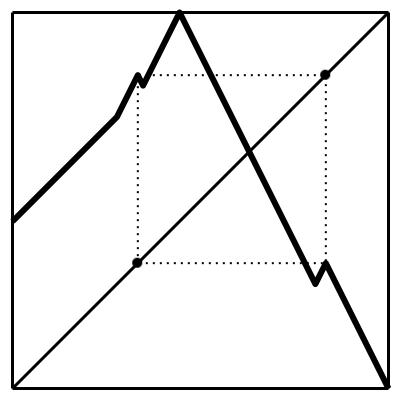} &&
\includegraphics[width=3.3cm]{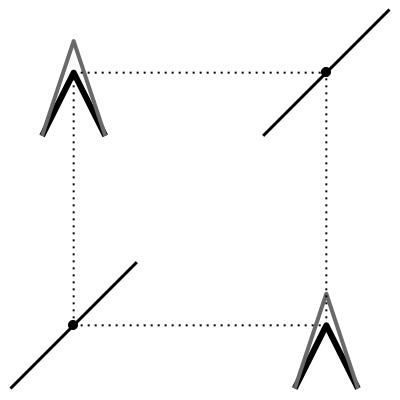} &&
\includegraphics[width=3.3cm]{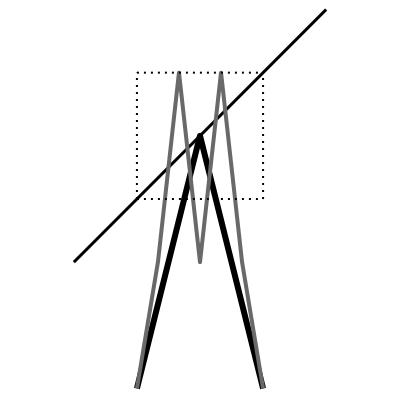} \tabularnewline
\tiny \rule[.6ex]{2em}{.2em} $f$ &&
\tiny \rule[.6ex]{2em}{.2em} $f$ \hspace{1em} 
\tiny {\color{darkgray} \rule[.6ex]{2em}{.15em}} $g_t$ &&
\tiny \rule[.6ex]{2em}{.2em} $f$ \hspace{1em}
\tiny {\color{darkgray} \rule[.6ex]{2em}{.15em}} $g^2_t$ \tabularnewline
\end{tabular}
\end{center}
\end{example}

\subsection*{Outline of the Paper}\strut\\
Section~\ref{sec:preimages} counts preimages to prove Theorem~\ref{th:preimages}. The central observation is that the strongly connected components of the Hofbauer diagram corresponding to a map $f\in\T$ are known to be positive recurrent.

Section~\ref{sec:equicontinuous} establishes our equicontinuity result, Theorem~\ref{th:equicontinuous}. The proof requires us to upgrade several facts about mixing piecewise monotone maps to use with perturbation.

Sections~\ref{sec:flatspots} and~\ref{sec:noflatspots} examine more closely what happens when we have a convergent sequence $g_n \rightrightarrows f$ in $\T_m$. We establish several properties of each (subsequential) limit $\psi$ of the corresponding homeomorphisms $\psi_n=\Psi(g_n)$.

Section~\ref{sec:wrapup} applies these properties to complete the proof of Theorem~\ref{th:main}.

Section~\ref{sec:further} leaves the reader with two open problems for further research.

\section{Counting Preimages}\label{sec:preimages}

\subsection*{Definitions}\strut\\
The \emph{Markov shift} associated to a countable (possibly finite) directed graph $\mathcal{G}$ with vertex set $\mathcal{V}$ is the set of all biinfinite paths on $\mathcal{G}$,
\begin{equation*}
\Sigma_{\mathcal{G}}=\left\{v\in\mathcal{V}^\mathbb{Z} \mid v_n \to v_{n+1} \text{ in }\mathcal{G}\text{ for all }n\in\mathbb{Z}\right\},
\end{equation*}
together with the shift map $(\sigma v)_n=v_{n+1}$. Its \emph{irreducible Markov subshifts} are the Markov shifts associated with the maximal strongly connected subgraphs of $\mathcal{G}$, where \emph{strongly connected} means that for each pair $(v,w)$ of vertices there is a path from $v$ to $w$. $\Sigma_{\mathcal{G}}$ becomes a measurable space when we equip it with the Borel sigma algebra, where the topology is induced from the product topology on $\mathcal{V}^\mathbb{Z}$. In the absence of compactness, the entropy of the shift is defined (following Gurevich \cite{G}) simply as the supremum of metric entropies
\begin{align*}
h(\Sigma_{\mathcal{G}}) &= 
\sup \{ h_\mu(\sigma) \,|\, \mu \text{ is a $\sigma$-invariant Borel probability measure on $\Sigma_{\mathcal{G}}$} \} \\
&= \sup \{ h_\mu(\sigma) \,|\, \mu \text{ is an ergodic $\sigma$-invariant Borel probability measure on $\Sigma_{\mathcal{G}}$} \}
\end{align*}

Given a strongly connected graph $\mathcal{G}$, the Markov shift $\Sigma_{\mathcal{G}}$ is called \emph{positive recurrent} if $\mathcal{G}$ has ``enough'' loops, so that if we fix a vertex $v$ and let $l_n$ count the number of length $n$ loops in $\mathcal{G}$ which start and end at $v$, we require that
$\limsup_{n\to\infty} l_n e^{-nh(\Sigma_{\mathcal{G}})} > 0$.
The strong connectedness of $\mathcal{G}$ guarantees that this property does not depend on the choice of the vertex $v$ (see, eg., \cite{VJ}).

\subsection*{Two shift spaces associated with an interval map}\strut\\
Let $f:[0,1]\to[0,1]$ be a transitive, piecewise monotone interval map. In particular, this implies that $f$ is surjective, piecewise strictly monotone, and has positive topological entropy. We follow the work of J. Buzzi and consider two shift spaces associated to $f$.

The first is a shift space usually called the \emph{symbolic dynamics of $f$}. It is a subshift in the alphabet $\mathcal{A}$ whose letters are the maximal open intervals on which $f$ is monotone. It is given by
\begin{equation*}
\Sigma=\{A\in\mathcal{A}^\mathbb{Z} \mid \forall{n\in\mathbb{Z}} \forall{k\geq0}\,\, A_n \cap f^{-1}(A_{n+1}) \cap \cdots \cap f^{-k}(A_{n+k})\neq\emptyset \}
\end{equation*}
together with the shift map $\sigma$. This shift space has the advantage that its alphabet is finite, but the disadvantage that it need not be a Markov shift.

The second is a Buzzi's variant of the Hofbauer shift. It is the Markov shift $\Sigmahat = \Sigma_{\mathcal{D}}$ associated to a certain directed graph $\mathcal{D}$ called the \emph{complete Markov diagram} of $\Sigma$. In the language of our original interval map, the definitions read as follows. A \emph{word} is a finite concatenation of letters from the alphabet $\mathcal{A}$. The set of points which are ``just finishing the itinerary'' given by a word is called the \emph{follower set} of the word,
\begin{equation*}
\fol{A_{-m}\cdots A_0} := f^m\left(A_{-m} \cap f^{-1}(A_{-m+1}) \cap \cdots \cap f^{-m}(A_0)\right).
\end{equation*}
By convention, the follower set of the empty word is the whole space $[0,1]$. A word is \emph{forbidden} if its follower set is empty. A \emph{constraint word} is a word $A_{-m}\cdots A_0$, $m\geq0$ whose follower set changes if we cross off the left-hand letter, i.e., such that
\begin{equation*}
\emptyset \neq \fol{A_{-m}\cdots A_0} \subsetneq \fol{A_{-m+1} \cdots A_0}
\end{equation*}
The collection of all constraint words will be denoted $\mathcal{C}$. Each nonforbidden word may be shortened to a constraint word by crossing off letters on the left, leaving behind the \emph{minimal} suffix with the same follower set. This motivates the definition
\begin{multline*}
\min(A_{-m}\cdots A_0)=A_{-k}\cdots A_0 \text{ if and only if} \\
 k\leq m \text{ and } \fol{ A_{-m}\cdots A_0} = \cdots = \fol{ A_{-k}\cdots A_0} \subsetneq \fol{ A_{-k+1}\cdots A_0}.
\end{multline*}

Finally we define the complete Markov diagram $\mathcal{D}$ as the directed graph with vertex set $\mathcal{C}$ and all arrows of the form
$
\alpha \to \min(\alpha A),
$
where $\alpha\in\mathcal{C}$, $A\in\mathcal{A}$, and $\alpha A$ is not forbidden.\footnote{In Hofbauer's original work the vertices are the follower sets, rather than the constraint words. We follow Buzzi's approach simply because his work contains the theorems we needed.}

\subsection*{Properties of these shift spaces}\strut\\
Now we can begin to exploit the connections between our transitive piecewise monotone map $f$, its symbolic dynamics $\Sigma$, and the Markov shift $\Sigmahat$ associated with its complete Markov diagram $\mathcal{D}$. We start by gathering together four known results:
\begin{itemize}
\item $h(f)=h(\Sigma)$.
\item $h(\Sigma)=h(\Sigmahat)$.
\item $\Sigmahat$ contains only finitely many positive-entropy irreducible Markov subshifts, and all of them are (strongly) positive recurrent.
\item The entropy of $\Sigmahat$ is the supremum of the entropies of its irreducible Markov subshifts.
\end{itemize}
The first result follows from Misiurewicz and Szlenk's characterization of the entropy of $f$ in terms of lap numbers, and noting that $f^n$ has as many laps as the number of length $n$ words appearing in the language of $\Sigma$ \cite{MS}. The next two results are Buzzi's, and are based on the fact that $\Sigma$, although not of finite type, is still a subshift of quasi-finite type \cite[Theorem 3 and Lemma 7]{B}. The fourth result follows because each ergodic invariant probability measure on a Markov shift is necessarily concentrated on one of its irreducible subshifts.

If we combine all four results, we may derive immediately
\begin{proposition}\label{prop:subgraph}
There exists a strongly connected subgraph $\mathcal{D}_0\subseteq\mathcal{D}$ whose associated Markov shift is positive recurrent and has the same entropy as $f$.
\end{proposition}

\subsection*{Follower Sets and Loops}\strut\\
Several properties of follower sets follow immediately from the definitions. For example, we can observe immediately that
\begin{itemize}
\item $\fol{A_{-m}\cdots A_0} \subseteq A_0$.
\end{itemize}
If we apply $n$ times the identity $A\cap f^{-1}(B)=f|_A^{-1}(B)$ to the definition of a follower set, we obtain
$\fol{A_{-m}\cdots A_0} = f^{m}\left[ f|_{A_{-m}}^{-1} f|_{A_{-m+1}}^{-1} \cdots f|_{A_{-1}}^{-1} A_0 \right]$. Since the monotone preimage of a connected set is connected, we can see that
\begin{itemize}
\item Each follower set is an interval.
\end{itemize}
Now consider the meaning of an arrow $\alpha\to\beta$ in $\mathcal{D}$. Write $\alpha=A_{-m}\cdots A_0$ and $\beta=\min(A_{-m}\cdots A_0 B)$. The definition of $\min(\cdot)$ gives $\fol{\beta}=\fol{A_{-m}\cdots A_0 B} = f^{m+1}\left[ A_{-m}\cap \cdots \cap f^{-m}(A_0) \cap f^{-m-1}(B) \right]$. On the other hand, the image of the follower set of $\alpha$ is $f(\fol{\alpha})=f^{m+1}\left[ A_{-m}\cap \cdots \cap f^{-m}(A_0) \right]$, which proves the implication
\begin{itemize}
\item If there is an arrow $\alpha\to\beta$, then $f(\fol{\alpha}) \supseteq \fol{\beta}$.
\end{itemize}
We also need a disjointness result. Fix a constraint word $\alpha\in\mathcal{C}$. If we form two more constraint words $\beta=\min(\alpha B)$ and $\gamma=\min(\alpha C)$ with $B\neq C$, then the disjointness of $B$ and $C$ gives disjointness of follower sets:
\begin{itemize}
\item If there are arrows $\alpha\to\beta$ and $\alpha\to\gamma$, $\beta\neq\gamma$, then $\fol{\beta}\cap\fol{\gamma}=\emptyset$.
\end{itemize}

If we combine all of these observations, we are ready to prove

\begin{lemma}\label{lem:fol-loops}
To each length-$n$ loop from a vertex $\alpha$ to itself in the complete Markov diagram, there corresponds a subinterval of $\fol{\alpha}$ which is mapped homeomorphically by $f^n$ onto $\fol{\alpha}$. Moreover, the subintervals corresponding to distinct length-$n$ loops are pairwise disjoint.
\end{lemma}
\begin{proof}
Let $\alpha=\alpha_0\to\alpha_1\to\cdots\to\alpha_n=\alpha$ be a loop. Put $I=\fol{\alpha_0}\cap f^{-1}\fol{\alpha_1}\cap \cdots \cap f^{-n}\fol{\alpha_n}$. Applying $n$ times the identity $A\cap f^{-1}(B)=f|_A^{-1}(B)$, this becomes
\begin{equation*}
I=f|_{\fol{\alpha_0}}^{-1} f|_{\fol{\alpha_1}}^{-1} \cdots f|_{\fol{\alpha_{n-1}}}^{-1} \fol{\alpha_n}.
\end{equation*}
Since the image of each follower set in the loop contains the next one, and since $f$ restricted to each follower set is monotone, we see that $f^n$ maps the interval $I$ monotonically and surjectively (i.e., homeomorphically) onto $\fol{\alpha_n}$.

Now let $\alpha=\beta_0\to\beta_1\to\cdots\to\beta_{n-1}\to\beta_n=\alpha$ be another length-$n$ loop which starts and ends at the same vertex $\alpha$, and let $J$ be the corresponding subinterval. Let $k$ be the minimum index so that $\alpha_k\neq\beta_k$. By hypothesis, $1\leq k\leq n-1$. Thus the two arrows $\alpha_{k-1}\to\alpha_k$ and $\beta_{k-1}\to\beta_k$ originate from the same vertex, so that $\fol{\alpha_k}\cap\fol{\beta_k}=\emptyset$. But $f^k(I)\subseteq\fol{\alpha_k}$ and $f^k(J)\subseteq\fol{\beta_k}$. Therefore $I\cap J=\emptyset$ also.
\end{proof}

\subsection*{Counting Preimages}\strut\\
We now have all the tools we need to prove Theorem~\ref{th:preimages}, which we restate here for the reader's convenience.

\begin{thm1}
Let $f\in\T$ and fix $x\in[0,1]$. Then
\begin{equation*}
\limsup_{n\to\infty} \frac{\# f^{-n}(x)}{e^{n h(f)}} > 0.
\end{equation*}
\end{thm1}
\begin{proof}
Let $\mathcal{D}$ be the complete Markov diagram of $f$ and $\mathcal{D}_0$ the subgraph promised by Proposition~\ref{prop:subgraph}. Fix a vertex $\alpha=A_{-m}\cdots A_0 \in\mathcal{D}_0$. Under a transitive piecewise monotone interval map, every point has a dense set of preimages \cite[see Theorem 2.19 and Proposition 2.34]{Ru}. Therefore, we can find a natural number $n_0$ such that $f^{-n_0}(x)\cap\langle A_{-m}\cdots \underline{A_0}\rangle\neq\emptyset$. Let $l_n$ be the number of loops of length $n$ in $\mathcal{D}_0$ which start and end at $\alpha$. By Lemma~\ref{lem:fol-loops} we get $\#f^{-n-n_0}(x)\geq l_n$, so using positive recurrence we get
\begin{equation*}
\limsup_{n\to\infty} \frac{\# f^{-n}(x)}{e^{nh(f)}} \geq 
\limsup_{n\to\infty} \frac{l_{n-n_0}}{e^{nh(\Sigma_{\mathcal{D}_0})}} > 0. \qedhere
\end{equation*}
\end{proof}

\section{Equicontinuity}\label{sec:equicontinuous}

To prove our equicontinuity result, we need to understand the behavior of critical points under perturbation in the space $\T_m$. We will use the following notation for the $\zeta$-neighborhood of a point $f$ in $\T_m$,
\begin{equation*}
N_\zeta(f):=\{g\in\T_m \mid d(f,g)<\zeta\}.
\end{equation*}

The following lemma records the simple observation that critical points vary continuously under perturbation in $\T_m$.

\begin{lemma}\label{lem:perturbcrit}
\emph{``Critical points vary continuously''}\\
Fix $f\in\T_m$ and $\rho>0$. Then there exists $\zeta>0$ such that if $g\in N_\zeta(f)$, then there is a bijection $\Crit(f)\to\Crit(g)$, $c\mapsto c'$, such that $0'=0$, $1'=1$, and $|c-c'|<\rho$ for all $c\in\Crit(f)$.\end{lemma}
\begin{proof}
Choose $\rho'<\rho$ less than half the distance between any two adjacent points in $\Crit(f)$. Now choose $\zeta>0$ sufficiently small so that $\zeta<\frac12|f(c)-f(w)|$ for $c\in \Crit(f)\setminus\{0,1\}$, $w=c\pm\rho'$.

Suppose $g\in N_\zeta(f)$.  We construct the bijection $\Crit(f)\to\Crit(g)$. Put $0'=0$ and $1'=1$. Now let $c\in\Crit(f)\setminus\{0,1\}$. By the choice of $\zeta$, the values $g(c-\rho')$, $g(c+\rho')$ both lie on the same side of $g(c)$. Therefore $g$ has a critical point $c'$ satisfying $|c-c'|<\rho'<\rho$. The mapping so defined is injective by the choice of $\rho'$. Since $f,g$ have the same modality, it is also surjective.
\end{proof}

Now we give a known property of interval maps, as well as an upgraded version for use with perturbation. They concern the accessibility of endpoints of the interval from the interior of the interval.

\begin{lemma}\label{lem:onto}
\emph{ \cite[Lemma 15]{M2} ``Accessibility of endpoints''}\\
For all $f\in\T_m$, $f^2((0,1))=[0,1].$
\end{lemma}

\begin{lemma}\label{lem:ontopert}
\emph{``Equi-accessibility of endpoints''}\\
For all $f\in\T_m$ there exist $\rho,\zeta>0$ such that for all $g\in N_\zeta(f)$, $g^2([\rho,1-\rho])=[0,1].$
\end{lemma}

\begin{proof}
We use the notation of ``$\rho$-neighborhoods'' for points $a$ and sets $A$, 
\begin{gather*}
N_\rho(a):=\{x\in[0,1] \mid |x-a|<\rho\} \\
N_\rho(A):=\{x\in[0,1] \mid |x-a|<\rho \text{ for some } a\in A\}.
\end{gather*}
Choose $\rho>0$ small enough that 
\begin{align*}
N_\rho(a)\cap N_\rho(\Crit(f)\setminus\{a\})&=\emptyset, \text{ and}\\
N_\rho(f(N_\rho(a))) \cap N_\rho(\Crit(f)\setminus\{f(a)\})&=\emptyset, \text{  for }a=0,1.
\end{align*}
Find $\zeta<\rho$ corresponding to $f$ and $\rho$ in Lemma~\ref{lem:perturbcrit}. Fix $g\in N_\zeta(f)$. Then for both $a=0,1$,
\begin{enumerate}[label=\textnormal{(\roman*)}]
\item\label{it:one} The only critical point of $g$ in $N_\rho(a)$ is $a$, and
\item\label{it:two} $g$ has at most one critical point in $g(N_\rho(a))$ (i.e. $f(a)'$).
\end{enumerate}
Lemma~\ref{lem:onto} applied to $g$ gives a pair of points $x,y\not\in\{0,1\}$ with
\begin{equation*}
x\mapsto g(x)\mapsto g^2(x)=0 \quad \text{and} \quad y\mapsto g(y)\mapsto g^2(y)=1.
\end{equation*}
Any point which maps into an endpoint of the interval $[0,1]$ must be critical. Therefore $g(x),g(y)\in\Crit(g)$. We claim that at least one of the points $x,g(x)$ belongs to $[\rho,1-\rho]$. Indeed, if $g(x)$ does not, then~\ref{it:one} implies $g(x)\in\{0,1\}$. Then $x\in\Crit(g)$, and~\ref{it:one} implies $x\in[\rho,1-\rho]$. The same argument shows that at least one of the points $y,g(y)$ belongs to $[\rho,1-\rho]$.

We finish the proof in cases.\\
\emph{Case 1:} $x,y\in[\rho,1-\rho]$. Then $g^2([\rho,1-\rho])=[0,1]$.\\
\emph{Case 2:} $g(x),g(y)\in[\rho,1-\rho]$. Then $g([\rho,1-\rho])=[0,1]$, so also $g^2([\rho,1-\rho])=[0,1]$.\\
\emph{Case 3:} $x,g(y)\in[\rho,1-\rho]$, $g(x),y\not\in[\rho,1-\rho]$. Then~\ref{it:one} implies $g(x)\in\{0,1\}$.
\begin{enumerate}[label=,leftmargin=0.5cm,topsep=0pt,partopsep=0pt]
\item \emph{Case 3a:} $g(x)=0$. Then $g([\rho,1-\rho])=[0,1]$, so also $g^2([\rho,1-\rho])=[0,1]$.
\item \emph{Case 3b:} $g(x)=1$. By the intermediate value theorem there is a point $z\in(x,1)$ with $g(z)=x$. Since $1,z$ both map to critical points,~\ref{it:two} implies $z\not\in N_\rho(1)$. So $g^2([\rho,1-\rho])$ contains both $0=g^2(x)$ and $1=g^2(z)$.
\end{enumerate}
\emph{Case 4:} $x,g(y)\notin[\rho,1-\rho]$, $g(x),y\in[\rho,1-\rho]$. Analogous to Case 3.
\end{proof}

The last ingredient we need is a kind of uniform locally eventually onto property. Unfortunately, it holds only in the subspace $\M_m\subset \T_m$ of maps which are topologically weak mixing. We remind the reader that $f$ is called (topologically) \emph{weak mixing} if for any nonempty open sets $A,B,C,D$ there exists $n\geq0$ such that $f^n(A)\cap B\neq\emptyset$ and $f^n(C)\cap D\neq\emptyset$ simultaneously. The next two lemmas say that a piecewise monotone weak mixing interval map $f$ is uniformly locally eventually onto, and the nearby maps in $\T_m$ are even equi-uniformly locally eventually onto. We can formulate analogous results when $f$ is transitive but not weak mixing, but the statements and proofs are technically involved and offer little insight, and are therefore deferred to the appendix.

\begin{lemma}\label{lem:unileo} 
\emph{\cite[Lemma 2.28]{Ru} ``Uniformly locally eventually onto''}\\
For all $f\in\M_m$ and all $\epsilon>0$ there exists $k\in\mathbb{N}$ such that for all $x,y\in[0,1]$, $$y-x>\epsilon \implies f^k([x,y])=[0,1].$$
\end{lemma}

\begin{lemma}\label{lem:equiunileo}
\emph{``Equi-uniformly locally eventually onto.''}\\
For all $f\in\M_m$ and all $\epsilon>0$ there exist $k\in\mathbb{N}$ and $\eta>0$ such that for all $g\in N_\eta(f)$ and all $x,y\in[0,1]$, $$y-x>\epsilon \implies g^{k+2}([x,y])=[0,1].$$
\end{lemma}

\begin{proof}
Let $f,\epsilon$ be given. Choose $k$ as in Lemma~\ref{lem:unileo} and $\rho,\zeta$ as in Lemma~\ref{lem:ontopert}. Finally, choose $\eta<\zeta$ small enough that $d(f,g)<\eta$ implies $d(f^k,g^k)<\rho$. For each $g\in N_\eta(f)$ and $x,y\in[0,1]$ with $y-x>\epsilon$, we get $$g^{k+2}([x,y]) = g^2(g^k([x,y])) \supseteq g^2([\rho,1-\rho]) = [0,1].$$
The containment follows because $f^k([x,y])=[0,1]$ and $g^k$ is $\rho$-close to $f^k$. The final equality is just Lemma~\ref{lem:ontopert}.
\end{proof}

Now we are ready to formulate and prove our equicontinuity result.

\begin{thm2}
If $K$ is a compact subset of $\T_m$, then $\Psi_m(K)$ is an equicontinuous family.
\end{thm2}

\begin{proof}
Topological entropy on the space of piecewise monotone maps of modality $m$ is not continuous, but jumps of entropy are bounded in the sense that $\limsup_{g\rightrightarrows f} h(g) \leq \max\{ h(f), \log 2 \}$, see \cite[Theorem 1]{M}. Thus each point $f\in K$ has a neighborhood on which $h$ is bounded above. By compactness, this implies that $h$ is bounded above on $K$, so put $\lambda = \sup \{ \exp h(f) \mid f\in K\}$. This $\lambda$ is a common Lipschitz constant for the constant slope models of all the maps in $K$.\footnote{Incidentally, this shows that $\Phi_m(K)$ is also an equicontinuous family. We remark, however, that the set of inverse homeomorphisms $\{\psi^{-1} | \psi\in\Psi(K)\}$ is not necessarily an equicontinuous family. That is why we defined the operator $\Psi$ in the direction we did, imagining a map's constant slope model as an extension rather than as a factor.}

Assume now that $K\subset \mathcal{M}_m$. The general case $K\subset \mathcal{T}_m$ is addressed in the appendix.

Fix $\epsilon>0$. The neighborhoods around each $f\in K$ guaranteed by Lemma~\ref{lem:equiunileo} form an open cover of $K$. Pass to a finite subcover and let $k_0$ be the maximum of the corresponding values of $k$. Fix $g\in K$. It belongs to one of those neighborhoods, so choosing $x,y\in[0,1]$ with $y-x>\epsilon$ we have 
$$g^{k_0+2}([x,y])=[0,1].$$
Write $\tilde{g}=\Phi_m(g)$ for the constant slope model and $\psi=\Psi_m(g)$ for the conjugating homeomorphism. Now we pass through the conjugacy $g=\psi \circ \tilde{g} \circ \psi^{-1}$ to obtain
$$\tilde{g}^{k_0+2}([\psi^{-1}x, \psi^{-1}y])=[0,1].$$
But $\tilde{g}$ has Lipschitz constant $\lambda$, so an interval which it stretches to length $1$ in $k_0+2$ steps must have length at least $\lambda^{-k_0-2}$. Writing $\delta=\lambda^{-k_0-2}$ we have proved that
$$ \forall \epsilon>0\,\, \exists \delta>0\,\, \forall \psi\in\Psi_m(K)\,\, \forall x,y\in[0,1]\, :\, y-x>\epsilon \implies \psi^{-1} y - \psi^{-1} x > \delta.$$
But this says exactly that the family $\Psi_m(K)$ is equicontinuous.
\end{proof}

\begin{corollary}\label{cor:equicontinuous}
Suppose $g_n \rightrightarrows f$ in $\T_m$. Then the sequence of homeomorphisms $\psi_n=\Psi_m(g_n)$ is equicontinuous.
\end{corollary}

\begin{proof}
The set $K=\{g_1, g_2, \ldots\} \cup \{ f \}$ is compact.
\end{proof}

\section{Flat Spots}\label{sec:flatspots}

Consider now the situation of a uniformly convergent sequence $g_n \rightrightarrows f$ in the space $\T_m$. We investigate the implications of this convergence on the corresponding constant slope models and conjugating homeomorphisms, as denoted in the following diagrams:
\begin{equation}\label{cd}
\begin{CD}
[0,1] @>\Phi(g_n)>> [0,1]\\
@V{\psi_n}VV @VV{\psi_n}V\\
[0,1] @>>g_n> [0,1]
\end{CD}
\qquad\qquad
\begin{CD}
[0,1] @>\Phi(f)>> [0,1]\\
@V{\phi}VV @VV{\phi}V\\
[0,1] @>>f> [0,1]
\end{CD}
\qquad\qquad
\arraycolsep=1.4pt
\begin{array}{rl}
\lambda_n&=\exp h(g_n) \\[1em]
\lambda&=\exp h(f)
\end{array}
\end{equation}

In this section we consider what happens if the homeomorphisms $\psi_n=\Psi(g_n)$ in~\eqref{cd} converge uniformly to a map $\psi$. A priori, $\psi$ is weakly monotone, but need not be a homeomorphism. It may have \emph{flat points:}
\begin{equation*}
\Flat(\psi):=\{x \,|\, \psi \text{ is constant on some neighborhood of } x\}
\end{equation*}
A \emph{flat value} is any element of the set $\psi(\Flat \psi)$.

The remaining properties of $\psi$ can be more or less guessed if we imagine that the maps $\Phi(g_n)$ are converging uniformly to some constant slope extension of $f$ with $\psi$ giving the semiconjugacy. In reality, we are able to prove the following results.

\begin{lemma}\label{lem:limpoints}
Suppose in~\eqref{cd} that $g_n\rightrightarrows f$ in $\T_m$ and $\psi_n\rightrightarrows \psi$.
If a sequence of points $y_{n}\rightarrow y$, then all limit points of the sequence $(\psi^{-1}_{n}(y_{n}))$ are in $\psi^{-1}(y)$.
\end{lemma}
\begin{proof}
Choose $\epsilon>0$ arbitrarily and let $x_{n_i} \rightarrow x'$, where $x_{n_{i}} = \psi^{-1}_{n_{i}}(y_{n_{i}})$.
Then for sufficiently large $i$ we have \\
$|y - y_{n_{i}}| < \epsilon$ by convergence $y_{n_{i}} \rightarrow y$;\\
$|\psi(x_{n_{i}}) - y_{n_{i}}| < \epsilon$ by uniform convergence $\psi_{n_{i}} \rightrightarrows \psi$ and $y_{n_{i}} = \psi_{n_{i}}(x_{n_{i}})$;\\
$|\psi(x_{n_{i}}) - \psi(x')| < \epsilon$ by convergence $x_{n_{i}} \rightarrow x'$ and continuity of $\psi$. \\
Combining these three inequalities $| y - \psi(x')| < 3 \epsilon$. Since $\epsilon$ was arbitrary, this shows $y = \psi(x')$, i.e., $x' \in \psi^{-1}(y)$.
\end{proof}

\begin{proposition}\label{prop:rectangles}
\emph{``Growth of rectangles''}\\
Suppose in~\eqref{cd} that $g_n\rightrightarrows f$ in $\T_m$, $\psi_n \rightrightarrows \psi$, and $\lambda_n\rightarrow\lambda'$. 
If $y<y'$ are two points such that $[y,y'] \cap \Crit(f)=\emptyset$ and none of $y,y',f(y),f(y')$ are flat values of $\psi$, then
\begin{equation*}
|\psi^{-1}\circ f (y')-\psi^{-1}\circ f(y)| = \lambda' \cdot |\psi^{-1}(y')-\psi^{-1}(y)|.
\end{equation*}
\end{proposition}
The meaning of the proposition is visualized in Figure~\ref{fig:rect}, which shows the graph of a purported limit map $\psi$. We imagine $f$ acting on the vertical axis. The dashed gray lines represent critical points of $f$. The meaning of the lemma, then, is that the rectangle determined by $f(y),f(y')$ is $\lambda'$ times wider than the rectangle determined by $y,y'$.
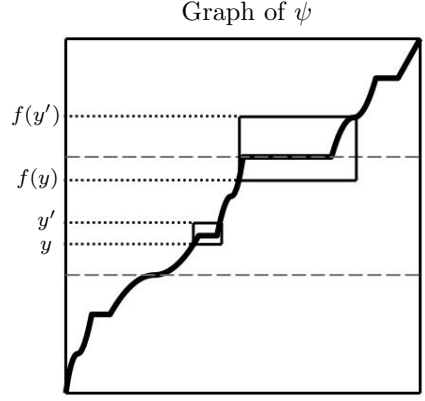
\begin{figure}[htb!!]
\footnotesize Graph of $\psi$\\
\includegraphics[width=5cm]{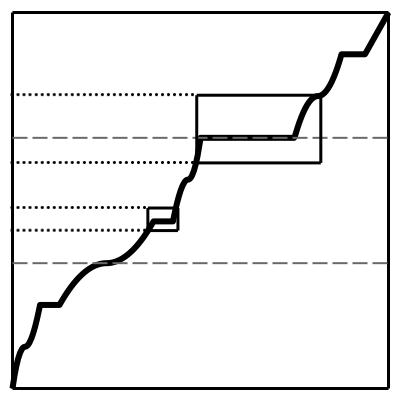}
\begin{picture}(0,0)
\put(-151,58){\tiny $y$}
\put(-152,67){\tiny $y'$}
\put(-160,83){\tiny $f(y)$}
\put(-162,107){\tiny $f(y')$}
\end{picture}
\caption{Growth of Rectangles}\label{fig:rect}
\end{figure}

\begin{proof}[Proof of Proposition~\ref{prop:rectangles}]
Using Lemma~\ref{lem:limpoints} and then (\ref{cd}) we have 
\begin{eqnarray*}
|\psi^{-1}\circ f(y')-\psi^{-1}\circ f(y)| &=& \lim_{n\rightarrow\infty} |\psi^{-1}_{n}\circ g_{n}(y') - \psi^{-1}_{n}\circ g_{n}(y)| \\
&=& \lim_{n\rightarrow\infty} |\Phi(g_n) \circ \psi^{-1}_{n}(y') - \Phi(g_n) \circ \psi^{-1}_{n}(y)|\\
\end{eqnarray*}
Now since $\Phi(g_n)$ are constant slope maps and since by Lemma~\ref{lem:perturbcrit} for sufficiently large $n$ both $\psi^{-1}_{n}(y'),$ $\psi^{-1}_{n}(y)$ lie in the same lap of $\Phi(g_n)$,
\begin{eqnarray*}
|\psi^{-1}\circ f(y')-\psi^{-1}\circ f(y)| &=& \lim_{n\rightarrow\infty} \lambda_{n} \cdot |\psi^{-1}_{n}(y') - \psi^{-1}_{n}(y)|.
\end{eqnarray*}
We finish by using Lemma~\ref{lem:limpoints} again
\begin{eqnarray*}
|\psi^{-1}\circ f(y')-\psi^{-1}\circ f(y)| &=& \lambda' \cdot |\psi^{-1}(y') - \psi^{-1}(y)|.
\end{eqnarray*}
\end{proof}

By letting $y,y'$ approach a point $b$ (possibly a flat value) from opposite sides, we prove

\begin{proposition}\label{prop:flat}
\emph{``Growth of flat spots''}\\
Suppose in~\eqref{cd} that $g_n\rightrightarrows f$ in $\T_m$, $\psi_n \rightrightarrows \psi$, and $\lambda_n\rightarrow\lambda'$.
\begin{enumerate}[label=\textnormal{(\alph*)},leftmargin=*]
\item\label{it:noncritical} If $b\notin\Crit(f)$, then $\len(\psi^{-1}(f(b)))=\lambda'\cdot\len(\psi^{-1}(b))$.
\item\label{it:endpoint} If $b\in\{0,1\}$, then $\len(\psi^{-1}(f(b)))\geq\lambda'\cdot\len(\psi^{-1}(b))$.
\end{enumerate}
\end{proposition}
\begin{proof}
(a) Let $(y_{i})$ be an increasing sequence and let $(y'_{i})$ be a decreasing sequence both with limit equal to $b$, moreover such that $[y_{1},y'_{1}] \cap \Crit(f) = \emptyset$ and none of $y,y',f(y),f(y')$ are flat values of $\psi$. Then $f(y_{i})$, $f(y'_{i})$ converge to $f(b)$ from opposite sides.
Now 
$$
  \len \psi^{-1}(b) = \lim_{i\rightarrow\infty} |\psi^{-1}(y'_{i}) - \psi^{-1}(y_{i})|.
$$ 
Using the previous proposition we arrive at 
$$
  \len \psi^{-1}(f(b)) = \lim_{i\rightarrow\infty} |\psi^{-1}(f(y'_{i})) - \psi^{-1}(f(y_{i}))|
  = \lambda' \cdot \lim_{i\rightarrow\infty} |\psi^{-1}(y'_{i}) - \psi^{-1}(y_{i})| = \lambda' \cdot \len \psi^{-1}(b).
$$

(b) We proof the assertion for the case $b=0$. 
Using Lemma~\ref{lem:limpoints} all limit points of the sequence $(\psi^{-1}_{n}(g_n(0)))$ belong to the compact set $\psi^{-1}(f(0))$. So replacing $(\psi_{n})$ by a subsequence and taking into account $\psi_n \rightrightarrows \psi$ 
we may assume that $\psi^{-1}_{n}(g_{n}(0)) \rightarrow x \in \psi^{-1}(f(0))$.
If $0<y$, $(0,y] \cap \Crit(f)=\emptyset$, and $\{y,f(y)\} \cap \psi(\Flat \psi) = \emptyset$, then 
\begin{eqnarray}
|\psi^{-1}\circ f(y) - x | &=& \lim_{n\rightarrow\infty} |\psi^{-1}_{n}\circ g_{n}(y) - \psi^{-1}_n\circ g_{n}(0)| \label{eq:prop-flat-a}\\
&=& \lim_{n\rightarrow\infty} |\Phi(g_n)\circ\psi^{-1}_{n}(y)-\Phi(g_n)\circ\psi^{-1}_{n}(0)|\label{eq:prop-flat-b}\\
&=& \lim_{n\rightarrow\infty} \lambda_{n} \cdot |\psi^{-1}_{n}(y) - \psi^{-1}_{n}(0)|\label{eq:prop-flat-c}\\
&=& \lambda' \cdot |\psi^{-1}(y) - 0|,\label{eq:prop-flat-d}
\end{eqnarray}
see Figure~\ref{fig:rect2}. Equation (\ref{eq:prop-flat-a}) holds by Lemma~\ref{lem:limpoints}, (\ref{eq:prop-flat-b}) holds by (\ref{cd}), (\ref{eq:prop-flat-c}) holds since all $\Phi(g_n)$ have constant slope, and again (\ref{eq:prop-flat-d}) holds by Lemma~\ref{lem:limpoints}.

\begin{figure}[htb!!]
\footnotesize Graph of $\psi$\\
\includegraphics[width=5cm]{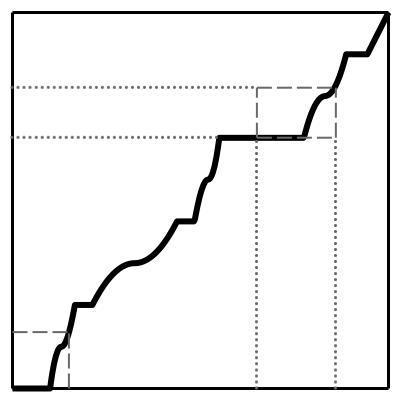}
\begin{picture}(0,0)
\put(-150,0){\tiny $0$}
\put(-150,23){\tiny $y$}
\put(-160,91){\tiny $f(0)$}
\put(-160,109){\tiny $f(y)$}
\put(-57,-4){\tiny $x$}
\end{picture}
\caption{Growth of Flat Spots}\label{fig:rect2}
\end{figure}
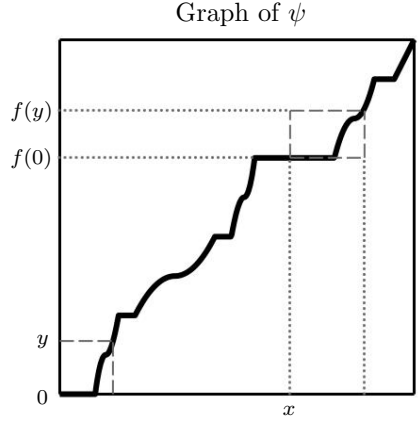
Now if $y_{i}$ converges to zero, $0<y_{i}$, $(0,y_{i}] \cap \Crit(f)=\emptyset$, and $\{y_{i},f(y_{i})\} \cap \psi(\Flat \psi) = \emptyset$ for all $i$, then $f(y_{i})$ converges to $f(0)$ from one side. 
Therefore 
$$
\len(\psi^{-1}(0)) = \lim_{i\rightarrow\infty} |\psi^{-1}(y_{i})|
$$
and
$$
\len(\psi^{-1}(f(0))) \geq \lim_{i\rightarrow\infty} |\psi^{-1}(f(y_{i})) - x| = 
 \lim_{i\rightarrow\infty} \lambda' |\psi^{-1}(y_{i}) - 0| = \lambda' \cdot \len(\psi^{-1}(0)). 
$$
\qedhere
\end{proof}

\begin{remark}\label{rem:c}
One might also consider the growth of flat spots of $\psi$ at critical points of $f$. Unfortunately, because of the folding, the strongest result we could prove was
\begin{enumerate}[leftmargin=*]
\item[(c)] If $b\in\Crit(f)$, then $\len(\psi^{-1}(f(b)))\geq\frac{\lambda'}{2}\cdot\len(\psi^{-1}(b))$.
\end{enumerate}
This result has no value for us, since $\lambda'/2$ may be less than $1$.
\end{remark}

\section{No Flat Spots}\label{sec:noflatspots}

\begin{proposition}\label{prop:homeomorphism}
Suppose in~\eqref{cd} that $g_n\rightrightarrows f$ in $\T_m$, $\psi_n\rightrightarrows \psi$, and $h(g_n)\to h(f)$. Then $\psi$ is a homeomorphism.
\end{proposition}

Here is the idea of the argument. We want to show that $\psi$ has no flat spots. Suppose to the contrary that $\psi$ collapses an interval of length $l>0$ to a single point $b$. As long as we avoid critical points, we can use Proposition~\ref{prop:flat} to produce other flat spots. The natural choice is to work forward along the orbit $b, f(b), \ldots$ until we find a flat spot of length $\lambda^n l > 1$, as in Table~\ref{tab:flat1}. Unfortunately, this argument, even using Remark~\ref{rem:c}, does not rule out flat spots on periodic orbits containing critical points. So instead we work backward along an orbit, producing flat spots at all points along a whole tree of preimages, as in Table~\ref{tab:flat2}.

\begin{table}[htb!!]
\begin{minipage}[b][11em][t]{.45\textwidth}
\small
\captionsetup{font=small,justification=centering}
\def\arraystretch{1.2}
\caption{Growth of flat spots, working forward}\label{tab:flat1}
\begin{tabular}{Y{.125} Y{0.50} Y{0.25}}
& Flat spot & Length \tabularnewline \cline{2-3}
& $b$ & $l$ \tabularnewline
& $f(b)$ & $\lambda l$ \tabularnewline
& $f^2(b)$ & $\lambda^2 l$ \tabularnewline
& $\vdots$ & $\vdots$ \tabularnewline
\end{tabular}
\end{minipage}%
\hspace{.05\textwidth}
\begin{minipage}[b][11em][t]{0.45\textwidth}
\small
\captionsetup{font=small,justification=centering}
\def\arraystretch{1.2}
\caption{\footnotesize Growth of flat spots, working backward}\label{tab:flat2}
\begin{tabular}{Y{0.67} Y{0.33}}
Flat spots & Length \tabularnewline \hline
$b$ & $l$ \tabularnewline
Each point of $f^{-1}(b)$ & $\lambda^{-1}l$ \tabularnewline
Each point of $f^{-2}(b)$ & $\lambda^{-2}l$ \tabularnewline
$\vdots$ & $\vdots$ \tabularnewline \midrule
\multicolumn{2}{c}{Total length: $l\sum\lambda^{-n}\#f^{-n}(b)=\infty$}
\end{tabular}
\end{minipage}
\end{table}

\begin{proof}[Proof of Proposition~\ref{prop:homeomorphism}]
Suppose to the contrary that $\Flat(\psi)\neq\emptyset$.
Choose $b\in\psi(\Flat \psi)$ and suppose for now that $b\notin\{0,1\}$.
Let $\Gamma = (V,E)$ be the smallest directed graph such that
$b \in V$ and if $v \in V$, $w\in[0,1]\setminus\Crit(f)$, and $f(w)=v$, then $w \in V$ and $(w,v) \in E$.
For any $v \in V$ write $|v|=\len \psi^{-1}(v)$ for the length of the flat spot.
By Proposition~\ref{prop:flat},
\begin{equation}\label{eq:prop-noflat-a}
\text{if } (w, v) \in E, \text{ then } |w| = \lambda^{-1} |v|.
\end{equation}
By supposition $|b|>0$ and consequently $|v|>0$ for all $v \in V$. 
From (\ref{eq:prop-noflat-a}) we get that $\Gamma$ contains no loops (because $\lambda=\exp h(f)>1$),
i.e. $\Gamma$ is a tree.

Unfortunately, we do not know if all preimages of $b$ are contained in this tree; we still need to avoid critical points.
Since $f$ is transitive and piecewise-monotone, each point of $(0,1)$ has non-critical preimages.
It follows that there is a backward orbit 
$$
 \dots \mapsto b_{-3} \mapsto b_{-2} \mapsto b_{-1} \mapsto b_{0}
$$
with $(b_{-n-1},b_{-n}) \in E$ for all $n \in \mathbb{N}$.
In other words, the set $B=\{b_{-1},b_{-2},\ldots\}$ contains no critical points of $f$.
Choose $n_0\geq1$ such that
\begin{equation}\label{eq:prop-noflat-b}
\begin{gathered}
\text{for each }c\in\Crit(f)\text{ whose forward orbit enters }B,\\
\text{ if the first entrance is at }b_{-m}\text{, then }m<n_0.
\end{gathered}
\end{equation}
We will prove that
\begin{equation}\label{eq:prop-noflat-c}
b_{-n_0}\text{ is not in the forward orbit of any critical point.}
\end{equation}

For suppose to the contrary that $f^t(c)=b_{-n_0}$, $c\in\Crit(f)$. 
Choose $r<t$ maximal such that $f^r(c)\in\Crit(f)$.
Choose $s>r$ minimal such that $f^s(c)\in B$.
By~\eqref{eq:prop-noflat-b}, $f^s(c)=b_{-m}$ with $m<n_0$.
\begin{equation*}
c \mapsto \cdots \mapsto 
\underbrace{f^r(c)}_{\mathclap{\substack{\text{last critical}\\\text{point before $b_{-n_0}$}}}}
\mapsto
\overbrace{
\cdots \mapsto
\underbrace{f^s(c)=b_{-m}}_{\mathclap{\substack{\text{first entrance to $B$}\\\text{after $f^r(c)$}}}}
 \mapsto \cdots
 }^{\substack{\text{No critical points}\\\text{by maximality of $r$}}}
 \mapsto
 \overbrace{
 f^t(c)=b_{-n_0} \mapsto \cdots \mapsto b_{-m}
 }^{\substack{\text{No critical points}\\\text{by construction of $B$}}}
\end{equation*}
This gives us a periodic point $b_{-m}\in V$ with no critical points in its periodic orbit.
This contradicts the fact that $\Gamma$ contains no loops, proving~\eqref{eq:prop-noflat-c}.

Now we add together the lengths of the flat spots at all preimages of $b_{-n_0}$.
By~\eqref{eq:prop-noflat-c} we have $f^{-i}(b_{-n_0})\subset V$ for all $i\in\mathbb{N}$.
By the absence of loops in $\Gamma$ we have $f^{-i}(b_{-n_0})\cap f^{-j}(b_{-n_0}) = \emptyset$ for $i\neq j$.
When we apply~\eqref{eq:prop-noflat-a} and Theorem~\ref{th:preimages} we get
$$
  \sum_{v\in V} |v| \geq |b_{-n_0}| \sum_{i\in \mathbb{N}} \lambda^{-i} \# f^{-i}(b_{-n_0}) = \infty.
$$
This is our contradiction, since the flat spots of $\psi$ are pairwise disjoint intervals contained in $[0,1]$.

Finally, consider the case when $b \in \{0,1\}$. 
If $f(b)$ belongs to $(0,1)$, then by Proposition~\ref{prop:flat}~\ref{it:endpoint} we may replace $b$ with $f(b)$ and proceed as before. If both $b, f(b) \in \{0,1\}$ but $f^2(b)$ belongs to $(0,1)$, then we use two applications of Proposition~\ref{prop:flat}~\ref{it:endpoint} to replace $b$ with $f^2(b)$ and proceed as before.

If all three points $b, f(b), f^2(b) \in \{0,1\}$, then two of these three points must coincide, which is impossible since by Proposition~\ref{prop:flat}~\ref{it:endpoint} we have
$\len(\psi^{-1}(b)) < \len(\psi^{-1}(f(b))) < \len(\psi^{-1}(f^{2}(b)))$.
\end{proof}

\section{Concluding Arguments}\label{sec:wrapup}

\begin{lemma}\label{lem:joint-cont}
Composition $(f_{1},f_{2}) \mapsto f_{1}\circ f_{2}$ is jointly continuous as a map $\mathcal{C}^0\times\mathcal{C}^0 \rightarrow \mathcal{C}^0$.
Inversion $\psi \mapsto \psi^{-1}$ is continuous as a map $\mathcal{H}^+ \to \mathcal{H}^+$.
\end{lemma}
\begin{proof}
We could not find the first statement anywhere, so we prove it here. Let $f_{1},f_{2}$ be given. Fix $\epsilon > 0$. 
Uniform continuity of $f_{1}$ give us $\eta$ such that if $|x-x'| < \eta$ then 
$|f_{1}(x)-f_{1}(x)'|<\frac{1}{2}\epsilon$.
If $f'_{1},f'_{2} \in \mathcal{C}^0$ are such that $d(f_{1},f'_{1}) < \frac12\epsilon$ and 
$d(f_{2},f'_{2}) < \eta$ then 
$$
|f'_{1}\circ f'_{2}(x) - f_{1}\circ f_{2}(x)| \leq 
|f'_{1}\circ f'_{2}(x) - f_{1}\circ f'_{2}(x)| + |f_{1}\circ f'_{2}(x) - f_{1}\circ f_{2}(x)| < 
\frac{\epsilon}{2} + \frac{\epsilon}{2} = \epsilon.
$$
Therefore $d(f'_{1}\circ f'_{2}, f_{1}\circ f_{2}) < \epsilon$.

The second statement is similar. A proof appears in \cite[Lemma 3.1 (c)]{KMS}
\end{proof}

\begin{lemma}\label{lem:checkCS}
If $f$ is piecewise monotone, $\psi\in\mathcal{H}^+$, $\lambda>1$, and
\begin{equation}
\label{checkCS} |\psi^{-1}\circ f(y') - \psi^{-1}\circ f(y)| = \lambda \cdot |\psi^{-1} y' - \psi^{-1} y|
\end{equation}
holds for all $y<y'$ such that $[y,y']\cap\Crit(f)=\emptyset$, then $\psi^{-1}\circ f\circ\psi$ has constant slope $\lambda$.
\end{lemma}
\begin{proof}
This is a simplified version of \cite[Lemma 4.6.4]{ALM}, but we include the proof for the reader's convenience. Write $\tilde{f}=\psi^{-1}\circ f\circ\psi$. Two points $x<x'$ belong to the interior of a lap of monotonicity of $\tilde{f}$ if and only if the corresponding points $y=\psi(x), y'=\psi(x')$ satisfy $[y,y']\cap\Crit(f)=\emptyset$, and in this case Equation~\eqref{checkCS} reduces to
$$ |\tilde{f}(x')-\tilde{f}(x)| = \lambda \cdot |x'-x|, $$
which says exactly that $\tilde{f}$ has constant slope $\lambda$.
\end{proof}

\begin{thm3}
If a sequence of maps $g_n\in\T_m$ converges uniformly to $f\in\T_m$ and if $h(g_n)\to h(f)$, then the constant slope models $\Phi(g_n)$ converge uniformly to $\Phi(f)$.
\end{thm3}

\begin{proof}
We need to show uniform convergence $\Phi(g_n) \rightrightarrows \Phi(f)$, i.e.
\begin{equation*}
\psi^{-1}_n \circ g_n \circ \psi_n \rightrightarrows \phi^{-1} \circ f \circ \phi,
\end{equation*}
where $\psi_n=\Psi(g_n)$, $\phi=\Psi(f)$ as in~\eqref{cd}. We already have $g_n\rightrightarrows f$ by hypothesis, so by Lemma~\ref{lem:joint-cont} it is enough to show uniform convergence $\psi_n \rightrightarrows \phi$.

By Theorem~\ref{th:equicontinuous}, $\{\psi_n\}$ is an equicontinuous family. Let $\psi$ be any subsequential limit $\psi_{n_i} \rightrightarrows \psi$. By Proposition~\ref{prop:homeomorphism}, $\psi$ is a homeomorphism, so we may consider the map $\psi^{-1} \circ f \circ \psi$. By Proposition~\ref{prop:rectangles} and Lemma~\ref{lem:checkCS} it has constant slope $\lambda=\exp h(f)=\lim_{n\to\infty} \exp h(g_n)$. By the uniqueness of constant slope models, $\psi^{-1}\circ f \circ \psi = \phi^{-1}\circ f \circ \phi$. By the uniqueness of the conjugating homeomorphism, $\psi=\phi$. Since this is true for every subsequential limit, it follows from equicontinuity that we have uniform convergence of the whole sequence $\psi_n \rightrightarrows \phi$, as desired.
\end{proof}

\section{Open Questions}\label{sec:further}

There remains still the interesting question of what happens when we have convergence of maps to a limit $f$ in $\T_m$ without convergence of entropy. What other maps besides $\Phi(f)$ can the constant slope models converge to?

\begin{question}
Characterize the set of all limit points of $\Phi(g)$ as $g\rightrightarrows f$ in $\T_m$. Is each limit point a constant slope extension of $f$? What other necessary or sufficient conditions are there? What slopes are possible?
\end{question}

We can also ask about stronger versions of continuity for the operator $\Phi_m$.

\begin{question}
Is $\Phi_m$ locally Lipschitz or Holder continuous at continuity points of the entropy? For $m\leq4$ can we get also global Lipschitz or Holder continuity?
\end{question}

\section*{Appendix}
{\footnotesize
The goal of this appendix is to finish the proof of Theorem~\ref{th:equicontinuous} by developing analogs of Lemmas~\ref{lem:unileo} and~\ref{lem:equiunileo} for maps which are transitive but not weak mixing. In essence, this is nothing more than a long technical obstacle, because for piecewise monotone interval maps, the various topological notions of a system's indecomposability are very closely related. In particular, it is known that
\begin{itemize}
\item The notions of topological weak mixing, topological strong mixing, and topological exactness (the locally eventually onto property) coincide.
\item If $f$ is transitive but not weak mixing, then it has a unique fixed point $e$, it interchanges $[0,e]$ with $[e,1]$, and both of the maps $f^2|_{[0,e]}$, $f^2|_{[e,1]}$ are weak mixing.
\end{itemize} 
A nice exposition of these results can be found in Ruette's textbook \cite[Proposition 2.34 and Theorem 2.19]{Ru}. We will use these results freely throughout this section.

\begin{lem9'}
Fix $f\in\T_m\setminus\M_m$ with unique fixed point $e$ and fix $\rho>0$. Then there exists $\zeta>0$ such that if $g\in N_\zeta(f)$, then $g$ has a unique fixed point $e'$ and $|e-e'|<\rho$.
\end{lem9'}
\begin{proof}
The proof is similar to the proof of Lemma~\ref{lem:perturbcrit}, and is left as an exercise for the reader.
\end{proof}

\begin{lem10'}
If $f\in\T_m$ has a unique fixed point $e$, then $f([0,e])\supseteq [e,1]$ and $f^2([0,e)) \supseteq [0,e]$.
\end{lem10'}
\begin{proof}
Notice that a map whose graph intersects the diagonal only once must lie above the diagonal to the left of this fixed point and below the diagonal to the right. Together with surjectivity, this implies that $f$ maps some point less than $e$ to $1$ and some point greater than $e$ to $0$. This gives the first result $f([0,e])\supseteq [e,1]$.

This also implies the containment $f^2([0,e))\supseteq [0,e)$. Now if $f$ is weak mixing, then $f^2$ is also weak mixing (iteration preserves the weak mixing property), hence $[0,e)$ is not invariant under $f^2$ and so the containment is strict $f^2([0,e))\supsetneq [0,e)$, yielding the second result.

Now suppose that $f\in\T_m\setminus\M_m$. Then the second iterate restricted to $[0,e]$ is locally eventually onto, so there is a minimal natural number $N$ with $f^{2N}([0,e))=[0,e]$. Surjectivity and the intermediate value theorem give us immediately the following containments
\begin{equation*}
[0,e) \subseteq f^2([0,e)) \subseteq f^4([0,e)) \subseteq \cdots,
\end{equation*}
and as soon as one of these containments is an equality, then all the following containments must be equalities also. Since $[0,e)$ omits only one point, we get $N\leq 1$
\end{proof}

\begin{lem11'}
Fix $f\in\T_m\setminus\M_m$ with unique fixed point $e$. Then there exist $\rho,\zeta>0$ such that if $g\in N_\zeta(f)$, then $g^4([\rho,e-\rho])\cup g^5([\rho,e-\rho]) = [0,1]$.
\end{lem11'}
\begin{proof}
We choose $\rho$ as in the proof of Lemma~\ref{lem:ontopert}, but with the following additional requirement. Since $e$ is fixed and not critical, we may choose $\rho$ so that
\begin{equation*}
N_\rho(f^i(N_\rho(e))) \cap N_\rho(\Crit f)=\emptyset \text{ for }i=0,1,2,3.
\end{equation*}
Now choose $\zeta$ answering to $\rho$ in Lemmas~\ref{lem:perturbcrit} and~\ref{lem:perturbcrit}$'$, and satisfying $\zeta<\rho$ and
\begin{equation*}
d(f,g)<\zeta \implies d(f^i,g^i)<\rho \text{  for }i=0,1,2,3.
\end{equation*}
Now fix $g\in N_\zeta(f)$. Our choice of $\rho$ and $\zeta$ gives
\begin{equation*}
g^i(N_\rho(e)) \cap \Crit(g)=\emptyset \text{ for }i=0,1,2,3.
\end{equation*}
If a point $c$ is critical for $g^2$, then either $c$ or $g(c)$ is critical for $g$. Therefore
\begin{equation}\label{nocrit}
N_\rho(e) \cap \Crit(g^2) = \emptyset \text{  and } g^2(N_\rho(e)) \cap \Crit(g^2) = \emptyset.
\end{equation}
By Lemma~\ref{lem:perturbcrit}$'$, $g$ has a unique fixed point $e'$ and $e'\in N_\rho(e)$. To finish the proof it suffices to show the following three containments:
\begin{enumerate}[label=\textnormal{(\roman*)}]
\item\label{it:step1} $g^2([\rho,e-\rho]) \supseteq [0,g^2(e-\rho)]$,
\item\label{it:step2} $g^2([0,g^2(e-\rho)]) \supseteq [0,e']$, and
\item\label{it:step3} $g([0,e']) \supseteq [e',1]$,
\end{enumerate}

By Lemma~\ref{lem:ontopert}, there is $x\in[\rho,1-\rho]$ with $g^2(x)=0$. Then $x\in\Crit(g^2)$, so by~\eqref{nocrit}, $x\notin N_\rho(e)$. Since $f^2([e,1])=[e,1]$ and $d(g^2,f^2)<\rho<e$, $x\notin [e,1]$ either. Therefore $x\in[\rho,e-\rho]$. This proves~\ref{it:step1}.

By Lemma~\ref{lem:onto}$'$ applied to $g$ we can find $y<e'$ with $g^2(y)=\max g^2|_{[0,e']}\geq e'$. Since this is a local maximum we have $y\in\Crit(g^2)$. By~\eqref{nocrit} neither $x$ nor $y$ belongs to $g^2(N_\rho(e))\supset (g^2(e-\rho),e')$. We have shown $0\leq x,y \leq g^2(e-\rho)$, which proves~\ref{it:step2}.

Lemma~\ref{lem:onto}$'$ applied to $g$ gives~\ref{it:step3}.
\end{proof}

\begin{lem12'}
For all $f\in\T_m$ and all $\epsilon>0$ there exists $k\in\mathbb{N}$ such that for all $x,y\in[0,1]$,
\begin{equation*}
y-x>\epsilon \implies \exists i\in\{0,1\} : f^{2k+i}([x,y]) = [0,e].
\end{equation*}
\end{lem12'}
\begin{proof}
 Let $e$ be the unique fixed point of $f$. We may apply Lemma~\ref{lem:unileo} to $f^2$ on each side of the phase space $[0,e]$, $[e,1]$ separately. Fixing $\epsilon>0$ we find a natural number $k$ so that if $y-x>\epsilon/2$ and $x,y\in[0,e]$, then $f^{2k}([x,y])=[0,e]$, whereas if $y-x>\epsilon/2$ and $x,y\in[e,1]$, then $f^{2k}([x,y])=[e,1]$. But $f([e,1])=[0,e]$. Now if $y-x>\epsilon$, then either $x,y$ are both on the same side of $e$, or one of $e-x$, $y-e$ is greater than $\epsilon/2$. In either case the result follows.
\end{proof}

\begin{lem13'}
For all $f\in\T_m\setminus\M_m$ and all $\epsilon>0$ there exist $k\in\mathbb{N}$ and $\eta>0$ such that for all $g\in N_\eta(f)$ and all $x,y\in[0,1]$,
$$y-x>\epsilon \implies \exists i\in\{0,1\} : g^{2k+i+4}([x,y])\cup g^{2k+i+5}([x,y]) =[0,1].$$
\end{lem13'}
\begin{proof}
Let $f,\epsilon$ be given. Let $e$ be the unique fixed point of $f$. Choose $k$ as in Lemma~\ref{lem:unileo}$'$ and $\rho,\zeta$ be as in Lemma~\ref{lem:ontopert}$'$. Choose $\eta<\zeta$ small enough that $d(f,g)<\eta$ implies $d(f^{2k+i},g^{2k+i})<\rho$ for $i=0,1$. Fix $g\in N_\eta(f)$ and $x,y\in [0,1]$ with $y-x>\epsilon$. Find $i\in\{0,1\}$ with $f^{2k+i}([x,y])=[0,e]$. Then
$$g^4(g^{2k+i}([x,y])) \cup g^5(g^{2k+i}([x,y])) \supseteq g^4([\rho,e-\rho]) \cup g^5([\rho,e-\rho]) = [0,1].$$ This completes the proof.
\end{proof}

Finally, we show how to complete the proof of Theorem~\ref{th:equicontinuous} in the general case $K\subset \T_m$, allowing for maps which are transitive but not weak mixing.

\begin{proof}[Proof of Theorem~\ref{th:equicontinuous}, (continued)]
Fix $\epsilon>0$. The neighborhoods around each $f\in K$ guaranteed by Lemmas~\ref{lem:equiunileo} and~\ref{lem:equiunileo}$'$ form an open cover of $K$. Pass to a finite subcover and let $k_0$ be the maximum of the corresponding values of $k$. Let $g\in K$. It belongs to one of those neighborhoods, so choosing $x,y\in[0,1]$ with $y-x>\epsilon$ we have 
$$\exists k\leq k_0\,\, \exists i\in\{0,1\} : g^{k+2}([x,y])=[0,1] \text{ or } g^{2k+i+4}([x,y])\cup g^{2k+i+5}([x,y])=[0,1].$$
In either case, further applications of the map $g$ to both sides of the equation yields
$$g^{2k_0+5}([x,y]) \cup g^{2k_0+6}([x,y]) = [0,1].$$
Write $\tilde{g}=\Phi_m(g)$ for the constant slope model and $\psi=\Psi_m(g)$ for the conjugating homeomorphism. Now we pass through the conjugacy $g=\psi \circ \tilde{g} \circ \psi^{-1}$ to obtain
$$\tilde{g}^{2k_0+5}([\psi^{-1}x, \psi^{-1}y]) \cup \tilde{g}^{2k_0+6}([\psi^{-1}x, \psi^{-1}y])=[0,1].$$
Thus at least one of the two intervals in the union above has length $\geq\frac12$. But $\tilde{g}$ has Lipschitz constant $\lambda$, so an interval which it stretches to length $\frac12$ in at most $k_0+6$ steps must have length at least $\frac12\lambda^{-k_0-6}$. Writing $\delta=\frac12\lambda^{-k_0-6}$ we have proved that
$$ \forall \epsilon>0\,\, \exists \delta>0\,\, \forall \psi\in\Psi_m(K)\,\, \forall x,y\in[0,1]\, :\, y-x>\epsilon \implies \psi^{-1} y - \psi^{-1} x > \delta.$$
But this says exactly that the family $\Psi_m(K)$ is equicontinuous.
\end{proof}

}

\end{document}